\documentclass[11pt]{amsart}
\usepackage[margin=37.5mm]{geometry}
\usepackage{amsmath,amssymb}
\usepackage{amsthm}

\newtheorem{thm}{Theorem}[section]

\newtheorem{lem}{Lemma}[section]

\newtheorem{defi}{Definition}[section]
\newtheorem{ex}{Example}[section]
\newtheorem{rem}{Remark}[section]
\newtheorem*{claim}{Claim}

\begin{document}

\title{On $C^0$-genericity of distributional chaos}
\author{Noriaki Kawaguchi$^\ast$}
\thanks{$^\ast$JSPS Research Fellow}
\address{Faculty of Science and Technology, Keio University, 3-14-1 Hiyoshi, Kohoku-ku, Yokohama, Kanagawa 223-8522, Japan}
\email{gknoriaki@gmail.com}

\begin{abstract}
Let $M$ be a compact smooth manifold without boundary. Based on results by Good and Meddaugh (2020), we prove that a strong distributional chaos is $C^0$-generic in the space of continuous self-maps (resp.\:homeomorphisms) of $M$. The results contain answers to questions by Li et al. (2016)  and Moothathu (2011) in the zero-dimensional case. A related counter-example on the chain components under shadowing is also given.
\end{abstract}

\maketitle

\markboth{NORIAKI KAWAGUCHI}{ON $C^0$-GENERICITY OF DISTRIBUTIONAL CHAOS}

\section{Introduction}

Throughout, $X$ denotes a compact metric space endowed with a metric $d$. We denote by $\mathcal{C}(X)$ (resp.\:$\mathcal{H}(X)$) the set of continuous self-maps (resp.\:homeomorphisms) of $X$. Let $d_{C^0}\colon\mathcal{C}(X)\times\mathcal{C}(X)\to[0,\infty)$ be the metric defined by
\[
d_{C^0}(f,g)=\sup_{x\in X}d(f(x),g(x))
\]
for $f,g\in\mathcal{C}(X)$. A metric $\hat{d}_{C^0}\colon\mathcal{H}(X)\times\mathcal{H}(X)\to[0,\infty)$ is given by
\[
\hat{d}_{C^0}(f,g)=\max\{d_{C^0}(f,g),d_{C^0}(f^{-1},g^{-1})\}
\]
for $f,g\in\mathcal{H}(X)$.  With respect to these metrics, $\mathcal{C}(X)$ and $\mathcal{H}(X)$ are complete metric spaces.

A subset $S$ of $X$ is called a {\em Mycielski set} if it is a countable union of Cantor sets.  Define $\mathcal{C}_{{\rm DC1}_\ast}(X)$ to be the set of $f\in\mathcal{C}(X)$ such that there is a Mycielski subset $S$ of $X$, which is distributionally $n$-$\delta_n$-scrambled for all $n\ge2$ for some $\delta_n>0$. Let
\[
\mathcal{H}_{{\rm DC1}_\ast}(X)=\mathcal{H}(X)\cap\mathcal{C}_{{\rm DC1}_\ast}(X).
\]
We say that a subset $F$ of a complete metric space $Z$ is {\em residual} if it contains a countable intersection of open and  dense subsets of $Z$. The aim of this paper is to prove the following theorem.

\begin{thm}
Given any compact smooth manifold $M$ without boundary, $\mathcal{C}_{{\rm DC1}_\ast}(M)$ is a residual subset of $\mathcal{C}(M)$, and if $\dim M>1$, then $\mathcal{H}_{{\rm DC1}_\ast}(M)$ is also a residual subset of $\mathcal{H}(M)$.
\end{thm}

We recall the definition of distributional $n$-chaos \cite{LO,TF}.

\begin{defi}
\normalfont
For $f\in\mathcal{C}(X)$, an $n$-tuple $(x_1,x_2,\dots,x_n)\in X^n$, $n\ge2$, is said to be {\em distributionally $n$-$\delta$-scrambled} for $\delta>0$ if
\begin{equation*}
\limsup_{m\to\infty}\frac{1}{m}|\{0\le k\le m-1\colon\max_{1\le i<j\le n}d(f^k(x_i),f^k(x_j))<\epsilon\}|=1
\end{equation*}
for all $\epsilon>0$, and
\begin{equation*}
\limsup_{m\to\infty}\frac{1}{m}|\{0\le k\le m-1\colon\min_{1\le i<j\le n}d(f^k(x_i),f^k(x_j))>\delta\}|=1.
\end{equation*}
Let ${\rm DC1}_n^\delta(X,f)$ denote the set of distributionally $n$-$\delta$-scrambled $n$-tuples and let ${\rm DC1}_n(X,f)=\bigcup_{\delta>0}{\rm DC1}_n^\delta(X,f)$.
A subset $S$ of  $X$ is said to be {\em distributionally $n$-scrambled} (resp.\:{\em $n$-$\delta$-scrambled}) if
\[
(x_1,x_2,\dots ,x_n)\in{\rm DC1}_n(X,f)\:\text{(resp.\:${\rm DC1}_n^\delta(X,f)$)}
\]
for any distinct $x_1,x_2,\dots,x_n\in S$. We say that $f$ exhibits the {\em distributional $n$-chaos of type 1} (${\rm DC1}_n$) if there is an uncountable distributionally $n$-scrambled subset of $X$.
\end{defi}

The notion of {\em distributional chaos} was introduced by Schweizer and Sm\'ital \cite{SS} as a refinement of Li-Yorke chaos for interval maps. It has three versions ${\rm DC}\beta_2$, $\beta\in\{1,2,3\}$, which are numbered in the order of decreasing strength. By definition, ${\rm DC1}_2$ is the strongest, and ${\rm DC2}_2$ is still stronger than Li-Yorke chaos. Let $h_{top}(f)$ denote the topological entropy of $f\in\mathcal{C}(X)$. For an interval map $f\in\mathcal{C}([0,1])$, all ${\rm DC}\beta_2$, $\beta\in\{1,2,3\}$, are equivalent to $h_{top}(f)>0$ \cite{SS} (see also \cite{Ru}). Since there is a Li-Yorke chaotic $f\in\mathcal{C}([0,1])$ with $h_{top}(f)=0$, ${\rm DC2}_2$ is strictly stronger than Li-Yorke chaos in general \cite{Smi, X} (see also \cite{Ru}). A well-known result in \cite{BGKM} showed that any $f\in\mathcal{C}(X)$ with $h_{top}(f)>0$ is Li-Yorke chaotic, and it was later improved by Downarowicz to ${\rm DC2}_2$ \cite{D}. However, Piku\l a \cite{Pik} constructed a subshift $(X,f)$ with $h_{top}(f)>0$ and
\[
{\rm DC1}_2(X,f)=\emptyset;
\]
therefore, $h_{top}(f)>0$ does not always imply ${\rm DC1}_2$. A Toeplitz subshift $(X,f)$ with $h_{top}(f)>0$ is a natural example of $f\in\mathcal{C}(X)$ with ${\rm DC1}_2(X,f)=\emptyset$ \cite{BS}. Thus, some additional assumptions besides $h_{top}(f)>0$ are needed to ensure ${\rm DC1}_2$ for a general $f\in\mathcal{C}(X)$.

{\em Shadowing} is a natural candidate for such an assumption. In \cite{LLT}, Li et al. proved that for any $f\in\mathcal{C}(X)$ with the shadowing property, there is a Mycielski subset $S$ of $X$, which is distributionally $n$-$\delta_n$-scrambled for all $n\ge2$ for some $\delta_n>0$, if one of the following properties holds: $(1)$ $f$ is non-periodic transitive and has a periodic point, or $(2)$ $f$ is non-trivial weakly mixing. Here, note that we have $h_{top}(f)>0$ in both cases. This result has been generalized in \cite{Ka2} by using a relation defined by Richeson and Wiseman \cite{RW}. In \cite{Ka1}, it was proved that for any $f\in\mathcal{C}(X)$ with the limit shadowing property, if $h_{top}(f)>0$, then $f$ exhibits ${\rm DC1}_2$. As a consequence, the set of $f\in\mathcal{C}(M)$ exhibiting ${\rm DC1}_2$, where $M$ is a compact topological manifold (possibly with boundary), is dense in $\mathcal{C}(M)$. However, since the limit shadowing has only been proved to be dense in $\mathcal{C}(M)$ \cite{MO}, it is still open whether ${\rm DC1}_2$ is generic or not. The above Theorem 1.1 solves this problem for any compact smooth manifold $M$ without boundary. Note that for every $n\ge2$, ${\rm DC1}_n$ does not necessarily imply ${\rm DC1}_{n+1}$ \cite{LO,TF}.

In outline, the proof of Theorem 1.1 goes as follows. In \cite{GM}, Good and Meddaugh found and investigated a basic relationship between the subshifts of finite type (abbrev.\:SFTs) and the shadowing. The following two lemmas are from \cite{GM}.

\begin{lem}
Let $\pi=(\pi_n^{n+1}\colon(X_{n+1},f_{n+1})\to(X_n,f_n))_{n\ge1}$ be an inverse sequence of equivariant maps and let $(X,f)=\lim_\pi(X_n,f_n)$. If $f_n\colon X_n\to X_n$ has the shadowing property for each $n\ge1$, and $\pi$ satisfies MLC, then $f$ has the shadowing property.
\end{lem}

\begin{lem}
Let $f\colon X\to X$ be a continuous map with the shadowing property. If $\dim X=0$, then there is an inverse sequence of equivariant maps
\[
\pi=(\pi_n^{n+1}\colon(X_{n+1},f_{n+1})\to(X_n,f_n))_{n\ge1}
\]
such that the following properties hold:
\begin{itemize}
\item[(1)] $\pi$ satisfies MLC,
\item[(2)] $(X_n,f_n)$ is a SFT for each $n\ge1$,
\item[(3)] $(X,f)$ is topologically conjugate to $\lim_{\pi}(X_n,f_n)$.
\end{itemize}
\end{lem}

Note that these results concern the so-called {\em Mittag-Leffler Condition} (MLC) of an inverse sequence of equivariant maps. Most part of this paper is devoted to a study of MLC with focus on the structure of chain components. By using the above lemmas and a method in \cite{Ka2} with Mycielski's theorem, we prove the following lemma. Here, $\mathcal{D}(f)$ is the partition of $X$ with respect to the equivalence relation $\sim_f$ defined by Richeson and Wiseman (see Section 2.2 for details). 

\begin{lem}
Let $f\colon X\to X$ be a transitive continuous map with the shadowing property. If $\dim X=0$ and $h_{top}(f)>0$, then there is $D\in\mathcal{D}(f)$ such that $D$ contains a dense Mycielski subset $S$, which is distributionally $n$-$\delta_n$-scrambled for all $n\ge2$ for some $\delta_n>0$.
\end{lem}

This lemma gives an answer to a question by Li et al. \cite{LLT} in the zero-dimensional case. By dropping the transitivity assumption through Lemma 4.1, we obtain the following theorem.

\begin{thm}
Let $f\colon X\to X$ be a continuous map with the shadowing property. If $\dim X=0$ and $h_{top}(f)>0$, then there is a Mycielski subset $S$ of $X$, which is distributionally $n$-$\delta_n$-scrambled for all $n\ge2$ for some $\delta_n>0$.
\end{thm}

Let
\begin{itemize}
\item $\mathcal{C}_{sh}(X)=\{f\in\mathcal{C}(X)\colon\text{$f$ has the shadowing property}\}$,
\item $\mathcal{C}_{cr^0}(X)=\{f\in\mathcal{C}(X)\colon\text{$\dim CR(f)=0$}\}$,
\item $\mathcal{C}_{h>0}(X)=\{f\in\mathcal{C}(X)\colon\text{$h_{top}(f)>0$}\}$,
\end{itemize}
and let $\mathcal{H}_\sigma(X)=\mathcal{H}(X)\cap\mathcal{C}_\sigma(X)$ for $\sigma\in\{sh,cr^0, {h>0}\}$. Note that for any
\[
f\in\mathcal{C}_{sh}(X)\cap\mathcal{C}_{cr^0}(X)\cap\mathcal{C}_{h>0}(X),
\]
the restriction $f|_{CR(f)}\colon CR(f)\to CR(f)$ has the following properties:
\begin{itemize}
\item the shadowing property,
\item $\dim CR(f)=0$,
\item $h_{top}(f|_{CR(f)})=h_{top}(f)>0$.
\end{itemize}
By applying Theorem 1.2 to $f|_{CR(f)}$, we obtain $f\in\mathcal{C}_{{\rm DC1}_\ast}(X)$; therefore, 
\[
\mathcal{C}_{sh}(X)\cap\mathcal{C}_{cr^0}(X)\cap\mathcal{C}_{h>0}(X)\subset\mathcal{C}_{{\rm DC1}_\ast}(X)
\]
and so
\[
\mathcal{H}_{sh}(X)\cap\mathcal{H}_{cr^0}(X)\cap\mathcal{H}_{h>0}(X)\subset\mathcal{H}_{{\rm DC1}_\ast}(X).
\]
Let $M$ be a compact smooth manifold without boundary. Then, Theorem 1.1 follows from those claims and the following results in the literature:
\begin{itemize}
\item shadowing
\begin{itemize}
\item $\mathcal{C}_{sh}(M)$ is a residual subset of $\mathcal{C}(M)$ \cite{MO},
\item $\mathcal{H}_{sh}(M)$ is a residual subset of $\mathcal{H}(M)$ \cite{PP},
\end{itemize}
\item chain recurrence
\begin{itemize}
\item $\mathcal{C}_{cr^0}(M)$ is a residual subset of $\mathcal{C}(M)$ \cite{KOU},
\item $\mathcal{H}_{cr^0}(M)$ is a residual subset of $\mathcal{H}(M)$ \cite{AHK},
\end{itemize}
\item topological entropy
\begin{itemize}
\item $\mathcal{C}_{h>0}(M)$ is a residual subset of $\mathcal{C}(M)$ \cite{Y},
\item If $\dim M>1$, then $\mathcal{H}_{h>0}(M)$ is a residual subset of $\mathcal{H}(M)$ \cite{Y}.
\end{itemize}
\end{itemize}

The proof also gives an insight into how the distributionally scrambled sets exist in the chain recurrent set. As in the proof of Lemma 4.1, any chain component with positive topological entropy is approximated by one with the shadowing property. Note that the component is partitioned into the equivalence classes of a relation by Richeson and Wiseman. Since the quotient dynamics with respect to the relation is minimal \cite{RW}, Lemma 1.3 implies that dense equivalence classes densely contain distributionally scrambled Mycielski sets within them. It deepens the understanding about the chaotic aspect of the $C^0$-generic dynamics on manifolds.

This paper consists of six sections. The basic notations, definitions, and facts are briefly collected in Section 2. In Section 3, we prove some preparatory lemmas. In Section 4, we prove Lemma 4.1 to reduce Theorem 1.2 to Lemma 1.3. Lemma 1.3 is proved in Section 5. In Section 6, as a bi-product of the proof of Lemma 4.1, we answer a question by Moothathu \cite{Moo} in the zero-dimensional case, and give a related counter-example showing that the chain components with the shadowing property can be relatively few.

\section{Preliminaries}

In this section, we collect some basic definitions, notations, facts, and prove some lemmas which will be used in the sequel.

\subsection{{\it Chains, cycles, pseudo-orbits, and the shadowing property}}

Given a continuous map $f\colon X\to X$, a finite sequence $(x_i)_{i=0}^{k}$ of points in $X$, where $k>0$ is a positive integer, is called a {\em $\delta$-chain} of $f$ if $d(f(x_i),x_{i+1})\le\delta$ for every $0\le i\le k-1$. A $\delta$-chain $(x_i)_{i=0}^{k}$ of $f$ is said to be a {\em $\delta$-cycle} of $f$ if $x_0=x_k$. Let $\xi=(x_i)_{i\ge0}$ be a sequence of points in $X$. For $\delta>0$, $\xi$ is called a {\em $\delta$-pseudo orbit} of $f$ if $d(f(x_i),x_{i+1})\le\delta$ for all $i\ge0$. For $\epsilon>0$, $\xi$ is said to be {\em $\epsilon$-shadowed} by $x\in X$ if $d(f^i(x),x_i)\leq \epsilon$ for all $i\ge 0$. We say that $f$ has the {\em shadowing property} if for any $\epsilon>0$, there is $\delta>0$ such that every $\delta$-pseudo orbit of $f$ is $\epsilon$-shadowed by some point of $X$.

\subsection{{\it Chain components and a relation}}

\subsubsection{{\it Chain recurrence and the chain transitivity}}

Given a continuous map $f\colon X\to X$, a point $x\in X$ is called a {\em chain recurrent point} for $f$ if for any $\delta>0$, there is a $\delta$-cycle $(x_i)_{i=0}^{k}$ of $f$ with $x_0=x_k=x$. We denote by $CR(f)$ the set of chain recurrent points for $f$. It is a closed $f$-invariant subset of $X$, and the restriction $f|_{CR(f)}\colon CR(f)\to CR(f)$ satisfies $CR(f|_{CR(f)})=CR(f)$. It is known that if $f$ has the shadowing property, then so does $f|_{CR(f)}$ \cite{Moo}. We call $f$ {\em chain recurrent} if $X=CR(f)$. For any $x,y\in X$ and $\delta>0$, the notation $x\rightarrow_{f,\delta} y$ means that there is a $\delta$-chain $(x_i)_{i=0}^k$ of $f$ with $x_0=x$ and $x_k=y$. Then, $f$ is said to be {\em chain transitive} if $x\rightarrow_{f,\delta} y$ for any $x,y\in X$ and $\delta>0$. We say that $f$ is {\em transitive} if for any two non-empty open subsets $U$, $V$of $X$, there is $n>0$ such that $f^n(U)\cap V\ne\emptyset$. If $f$ is transitive, then $f$ is chain transitive, and the converse holds when $f$ has the shadowing property.
 
\subsubsection{{\it Chain components}}

For any continuous map $f\colon X\to X$, $CR(f)$ admits a decomposition with respect to a relation $\leftrightarrow_f$ in $CR(f)^2=CR(f)\times CR(f)$ defined as follows: for any $x,y\in CR(f)$, $x\leftrightarrow_f y$ iff  $x\rightarrow_{f,\delta} y$ and $y\rightarrow_{f,\delta} x$ for every $\delta>0$. Note that $\leftrightarrow_f$ is a closed $(f\times f)$-invariant equivalence relation in $CR(f)^2$. An equivalence class $C$ of $\leftrightarrow_f$ is called a {\em chain component} for $f$. We denote by $\mathcal{C}(f)$ the set of chain components for $f$. Then, the following properties hold:
\begin{itemize}
\item[(1)] $CR(f)=\bigsqcup_{C\in\mathcal{C}(f)}C,$
\item[(2)] Every $C\in\mathcal{C}(f)$ is a closed $f$-invariant subset of $CR(f)$,
\item[(3)] $f|_C\colon C\to C$ is chain transitive for all $C\in\mathcal{C}(f)$.
\end{itemize}
Note that $f$ is chain transitive iff $f$ is chain recurrent and satisfies $\mathcal{C}(f)=\{X\}$.

\subsubsection{{\it A relation}}

Let $f\colon X\to X$ be a chain transitive map. For $\delta>0$ and a $\delta$-cycle $\gamma=(x_i)_{i=0}^k$ of $f$, $k$ is called the {\em length} of $\gamma$. Let $m=m(f,\delta)>0$ be the greatest common divisor of the lengths of $\delta$-cycles of $f$. We define a relation $\sim_{f,\delta}$ in $X^2$ by for any $x,y\in X$, $x\sim_{f,\delta}y$ iff there is a $\delta$-chain $(x_i)_{i=0}^k$ of $f$ with $x_0=x$, $x_k=y$, and $m|k$. Then, the following properties hold:
\begin{itemize}
\item[(1)] $\sim_{f,\delta}$ is an open and closed $(f\times f)$-invariant equivalence relation in $X^2$,
\item[(2)] For any $x\in X$ and $n\ge0$, $x\sim_{f,\delta}f^{mn}(x)$,
\item[(3)] There exists $N>0$ such that for any $x,y\in X$ with $x\sim_{f,\delta}y$ and $n\ge N$, there is a $\delta$-chain $(x_i)_{i=0}^k$ of $f$ with $x_0=x$, $x_k=y$, and $k=mn$.
\end{itemize}

Following \cite{RW}, define a relation $\sim_f$ in $X^2$ by for any $x,y\in X$, $x\sim_f y$ iff $x\sim_{f,\delta}y$ for every $\delta>0$. This is a closed $(f\times f)$-invariant equivalence relation in $X^2$. We denote by $\mathcal{D}(f)$ the set of equivalence classes of $\sim_f$. It gives a closed partition of $X$. A pair $(x,y)\in X^2$ is said to be {\em chain proximal} if for any $\delta>0$, there is a pair $((x_i)_{i=0}^k,(y_i)_{i=0}^k)$ of $\delta$-chains of $f$ such that $(x_0,y_0)=(x,y)$ and $x_k=y_k$. As claimed in \cite[Remark 8]{RW}, for any $(x,y)\in X^2$, $(x,y)$ is chain proximal iff $x\sim_f y$.

\subsection{{\it Inverse limit}}

\subsubsection{{\it Inverse limit spaces}}

Given an inverse sequence of continuous maps
\[
\pi=(\pi_n^{n+1}\colon X_{n+1}\to X_n)_{n\ge1},
\]
where $(X_n)_{n\ge1}$ is a sequence of compact metric spaces, define $\pi_n^m\colon X_m\to X_n$ by
\begin{equation*}
\pi_n^m=
\begin{cases}
id_{X_n}&\text{if $m=n$}\\
\pi_n^{n+1}\circ\pi_{n+1}^{n+2}\circ\cdots\circ\pi_{m-1}^m&\text{if $m>n$}
\end{cases}
\end{equation*}
for all $m\ge n\ge1$. Note that $\pi_n^l=\pi_n^m\circ\pi_m^l$ for any $l\ge m\ge n\ge1$. 
The {\em inverse limit space} $X=\lim_\pi X_n$ is defined by
\[
X=\{x=(x_n)_{n\ge1}\in\prod_{n\ge1}X_n\colon\pi_n^{n+1}(x_{n+1})=x_n,\forall n\ge1\},
\]
which is a compact metric space.

For any $n\ge1$, note that $\pi_n^m(X_m)\supset\pi_n^{m+1}(X_{m+1})$ for every $m\ge n$, and let
\[
\hat{X}_n=\bigcap_{m\ge n}\pi_n^m(X_m).
\]
By the compactness, we easily see that for any $n\ge1$ and $x\in X_n$, $x\in\hat{X}_n$ iff there is a sequence
\[
(x_m)_{m\ge n}\in\prod_{m\ge n}X_m
\]
with $\pi_m^{m+1}(x_{m+1})=x_m$ for every $m\ge n$. For each $n\ge1$, $\pi_n^{n+1}(\hat{X}_{n+1})=\hat{X}_n$, i.e., $\hat{\pi}_n^{n+1}=(\pi_n^{n+1})|_{\hat{X}_{n+1}}\colon\hat{X}_{n+1}\to\hat{X}_n$ is surjective. Let
\[
\hat{\pi}=(\hat{\pi}_n^{n+1}\colon\hat{X}_{n+1}\to\hat{X}_n)_{n\ge1}
\]
and $\hat{X}=\lim_{\hat{\pi}}\hat{X}_n$. Since any $x=(x_n)_{n\ge1}\in X$ satisfies $x_n\in\hat{X}_n$ for every $n\ge1$, we see that the inclusion $i\colon\hat{X}\to X$ is a homeomorphism.

\subsubsection{{\it The Mittag-Leffler Condition}}

Given $\pi=(\pi_n^{n+1}\colon X_{n+1}\to X_n)_{n\ge1}$, an inverse sequence of continuous maps, $\pi$ is said to satisfy the {\em Mittag-Leffler Condition} (MLC) if for any $n\ge1$, there is $N\ge n$ such that $\pi_n^N(X_N)=\pi_n^m(X_m)$ for all $m\ge N$. We say that $\pi$ satisfies MLC(1) if
\[
\pi_n^{n+1}(X_{n+1})=\pi_n^{n+2}(X_{n+2})
\]
for any $n\ge1$.

\begin{lem}
Let $\pi=(\pi_n^{n+1}\colon X_{n+1}\to X_n)_{n\ge1}$ be an inverse sequence of continuous maps. Then, the following properties are equivalent:
\begin{itemize}
\item[(1)] $\pi$ satisfies MLC(1),
\item[(2)] For any $n\ge1$ and $m\ge n+1$, $\pi_n^{n+1}(X_{n+1})=\pi_n^m(X_m)$,
\item[(3)] For every $n\ge1$, $\hat{X}_n=\pi_n^{n+1}(X_n)$.
\end{itemize}
\end{lem}

\begin{proof}
The implication $(1)\Rightarrow(2)$: we use an induction on $m$. For $m=n+1$, $\pi_n^{n+1}(X_{n+1})=\pi_n^m(X_m)$ is trivially true. Assume $\pi_n^{n+1}(X_{n+1})=\pi_n^m(X_m)$ for some $m\ge n+1$. Then, we have
\[
\pi_n^{m+1}(X_{m+1})=\pi_n^{m-1}(\pi_{m-1}^{m+1}(X_{m+1}))=\pi_n^{m-1}(\pi_{m-1}^m(X_m))=\pi_n^m(X_m)=\pi_n^{n+1}(X_n),
\]
completing the induction.
$ $\newline
$(2)\Rightarrow(1)$: Put $m=n+2$ in (2).
$ $\newline
$(2)\Rightarrow(3)$: (2) implies
\[
\hat{X}_n=\pi_n^n(X_n)\cap\bigcap_{m\ge n+1}\pi_n^m(X_m)=X_n\cap\pi_n^{n+1}(X_{n+1})=\pi_n^{n+1}(X_{n+1})
\]
for every $n\ge1$.
$ $\newline
$(3)\Rightarrow(2)$: Since $\pi_n^{n+1}(X_{n+1})\supset\pi_n^{n+2}(X_{n+2})\supset\cdots\supset\hat{X}_n$, $\hat{X}_n=\pi_n^{n+1}(X_{n+1})$ implies $\pi_n^m(X_m)=\hat{X}_n=\pi_n^{n+1}(X_{n+1})$ for any $m\ge n+1$, completing the proof.
\end{proof}

\begin{rem}
\normalfont
The property (2) in the above lemma implies that $\pi$ satisfies MLC.
\end{rem}

\begin{lem}
Let $\pi=(\pi_n^{n+1}\colon X_{n+1}\to X_n)_{n\ge1}$ be an inverse sequence of continuous maps. If $\pi$ satisfies MLC, then there is a sequence $1\le n(1)<n(2)<\cdots$ such that, letting $\pi'=(\pi_{n(j)}^{n(j+1)}\colon X_{n(j+1)}\to X_{n(j)})_{j\ge1}$, $\pi'$ satisfies MLC(1).
\end{lem}

\begin{proof}
Put $n(0)=1$. Inductively, define a sequence $1=n(0)<n(1)<n(2)<\cdots$ as follows: given $j\ge0$ and $n(j)$, take $n(j+1)>n(j)$ such that $\pi_{n(j)}^{n(j+1)}(X_{n(j+1)})=\pi_{n(j)}^m(X_m)$ for every $m\ge n(j+1)$. Then, for each $j\ge1$, $\pi_{n(j)}^{n(j+1)}(X_{n(j+1)})=\pi_{n(j)}^{n(j+2)}(X_{n(j+2)})$ since $n(j+2)>n(j+1)$, implying that $\pi'$ satisfies MLC(1).
\end{proof}

\subsubsection{{\it Equivariance, factor, and the topological conjugacy}}

Given two continuous maps $f\colon X\to X$ and $g\colon Y\to Y$, where $X$ and $Y$ are compact metric spaces, a continuous map $\pi\colon X\to Y$ is said to be {\em equivariant} if $g\circ\pi=\pi\circ f$, and such $\pi$ is also denoted as $\pi\colon(X,f)\to(Y,g)$. An equivariant map $\pi\colon(X,f)\to(Y,g)$ is called a {\em factor map} (resp.\:{\em topological conjugacy}) if it is surjective (resp.\:a homeomorphism). Two systems $(X,f)$ and $(Y,g)$ are said to be {\em topologically conjugate} if there is a topological conjugacy $h\colon(X,f)\to(Y,g)$.

\subsubsection{{\it Inverse limit systems}}

For an inverse sequence of equivariant maps
\[
\pi=(\pi_n^{n+1}\colon(X_{n+1},f_{n+1})\to(X_n,f_n))_{n\ge1},
\]
the {\em inverse limit system} $(X,f)=\lim_\pi(X_n,f_n)$ is well-defined by $X=\lim_\pi X_n$, and $f(x)=(f_n(x_n))_{n\ge1}$ for all $x=(x_n)_{n\ge1}\in X$.

For every $n\ge1$, note that $\hat{X}_n$ is a closed $f_n$-invariant subset of $X_n$, and let $\hat{f}_n=(f_n)|_{\hat{X}_n}\colon\hat{X}_n\to\hat{X}_n$. For all $n\ge 1$, $\hat{\pi}_n^{n+1}=(\pi_n^{n+1})|_{\hat{X}_{n+1}}\colon\hat{X}_{n+1}\to\hat{X}_n$ gives a factor map
\[
\hat{\pi}_n^{n+1}\colon(\hat{X}_{n+1},\hat{f}_{n+1})\to(\hat{X}_n,\hat{f}_n).
\]
Let 
\[
\hat{\pi}=(\hat{\pi}_n^{n+1}\colon(\hat{X}_{n+1},\hat{f}_{n+1})\to(\hat{X}_n,\hat{f}_n))_{n\ge1}
\]
and $(\hat{X},\hat{f})=\lim_{\hat{\pi}}(\hat{X}_n,\hat{f}_n)$. Then, the inclusion $i\colon\hat{X}\to X$ is a topological conjugacy $i\colon(\hat{X},\hat{f})\to(X,f)$.

\begin{lem}
Let $\pi=(\pi_n^{n+1}\colon(X_{n+1},f_{n+1})\to(X_n,f_n))_{n\ge1}$ be an inverse sequence of equivariant maps and let $(X,f)=\lim_\pi(X_n,f_n)$. If $f_n\colon X_n\to X_n$ is chain recurrent (resp.\:chain transitive) for each $n\ge1$, then $f\colon X\to X$ is chain recurrent (resp.\:chain transitive).
\end{lem}

\begin{proof}
Let $n\ge1$. Then, for any $m\ge n$, since $f_m\colon X_m\to X_m$  is chain recurrent (resp.\:chain transitive), and $\pi_n^m\colon(X_m,f_m)\to(X_n,f_n)$ is an equivariant map,
\[
(f_n)|_{\pi_n^m(X_m)}\colon\pi_n^m(X_m)\to\pi_n^m(X_m)
\]
is chain recurrent (resp.\:chain transitive). Because $\hat{X}_n=\bigcap_{m\ge n}\pi_n^m(X_m)$,
\[
\hat{f}_n\colon\hat{X}_n\to\hat{X}_n
\]
is chain recurrent (resp.\:chain transitive). Since
\[
\hat{\pi}=(\hat{\pi}_n^{n+1}\colon(\hat{X}_{n+1},\hat{f}_{n+1})\to(\hat{X}_n,\hat{f}_n))_{n\ge1}
\]
is a sequence of factor maps, we see that $\hat{f}\colon\hat{X}\to\hat{X}$ is chain recurrent (resp.\:chain transitive). Thus, $f\colon X\to X$ is chain recurrent (resp.\:chain transitive), because $(X,f)$ and $(\hat{X},\hat{f})$ are topologically conjugate, completing the proof.
\end{proof}

For an inverse sequence of equivariant maps
\[
\pi=(\pi_n^{n+1}\colon(X_{n+1},f_{n+1})\to(X_n,f_n))_{n\ge1}
\]
and a sequence $1\le n(1)<n(2)<\cdots$, $\pi_{n(j)}^{n(j+1)}\colon X_{n(j+1)}\to X_{n(j)}$ is an equivariant map for each $j\ge1$. Letting 
\[
\pi'=(\pi_{n(j)}^{n(j+1)}\colon(X_{n(j+1)},f_{n(j+1)})\to(X_{n(j)},f_{n(j)}))_{j\ge1}
\]
and $(Y,g)=\lim_{\pi'}(X_{n(j)},f_{n(j)})$, we have a topological conjugacy $h\colon(X,f)\to(Y,g)$ given by $h(x)=(x_{n(j)})_{j\ge1}$ for all $x=(x_n)_{n\ge1}\in X$. By this and Lemma 2.2, we obtain the following.

\begin{lem}
Let $\pi=(\pi_n^{n+1}\colon(X_{n+1},f_{n+1})\to(X_n,f_n))_{n\ge1}$ be an inverse sequence of equivariant maps. If $\pi$ satisfies MLC, then there is a sequence $1\le n(1)<n(2)<\cdots$ such that, letting $\pi'=(\pi_{n(j)}^{n(j+1)}\colon(X_{n(j+1)},f_{n(j+1)})\to(X_{n(j)},f_{n(j)}))_{j\ge1}$, $\pi'$ satisfies MLC(1), and $\lim_\pi(X_n,f_n)$ is topologically conjugate to $\lim_{\pi'}(X_{n(j)},f_{n(j)})$.
\end{lem}

\subsection{{\it Subshifts}}

\subsubsection{{\it Subshifts of finite type}}

Let $S$ be a finite set with the discrete topology. The {\em shift map} $\sigma\colon S^\mathbb{N}\to S^\mathbb{N}$ is defined by $\sigma(x)=(x_{n+1})_{n\ge1}$ for all $x=(x_n)_{n\ge1}\in S^\mathbb{N}$. Note that $\sigma$ is continuous with respect to the product topology of $S^\mathbb{N}$. The product space $S^\mathbb{N}$ (and also $(S^\mathbb{N},\sigma)$) is called the (one-sided) {\em full-shift} over $S$.  A closed $\sigma$-invariant subset $X$ of $S^\mathbb{N}$ (and also the subsystem $(X,\sigma|_X)$ of $(S^\mathbb{N},\sigma)$) is called a {\em subshift}. A subshift $X$ of $S^\mathbb{N}$ (and also $(X,\sigma|_X)$ of $(S^\mathbb{N},\sigma)$) is called a {\em subshift of finite type} (abbrev.\:SFT) if there are $N>0$ and $F\subset S^{N+1}$ such that for any $x=(x_n)_{n\ge1}\in S^\mathbb{N}$, $x\in X$ iff $(x_i,x_{i+1},\dots,x_{i+N})\in F$ for all $i\ge1$. The shift map $\sigma\colon S^\mathbb{N}\to S^\mathbb{N}$ is positively expansive and has the shadowing property. We know that a subshift $X$ of $S^\mathbb{N}$ is of finite type iff $\sigma|_X\colon X\to X$ has the shadowing property \cite{AHi}.

\subsubsection{{\it Some properties of SFTs}}

Let $(X,\sigma|_X)$ be a SFT (of some full-shift over $S$) and put $f=\sigma|_X$. Then, $f$ has the following properties:
\begin{itemize}
\item[(1)] $CR(f)=\overline{Per(f)}$, where $Per(f)$ denotes the set of periodic points for $f$,
\item[(2)] For any $x\in X$, there is $y\in CR(f)$ such that $\lim_{n\to\infty}d(f^n(x),f^n(y))=0$.
\end{itemize}
In fact, these two properties are consequences of the positive expansiveness and the shadowing property of $f\colon X\to X$. Since the restriction $f|_{CR(f)}\colon CR(f)\to CR(f)$ is surjective and positively expansive, it is c-expansive. Also, it has the shadowing property. Applying \cite[Theorem 3.4.4]{AHi} to $f|_{CR(f)}$, we obtain
\begin{itemize}
\item[(3)] There is a finite set $\mathcal{C}$ of clopen $f$-invariant subsets of $CR(f)$ such that
\[
CR(f)=\bigsqcup_{C\in\mathcal{C}}C,
\]
and $f|_{C}\colon C\to C$ is transitive for every
$C\in\mathcal{C}$.
\end{itemize}
An element of $\mathcal{C}$ is called a {\em basic set}. We easily see that $\mathcal{C}=\mathcal{C}(f)$, i.e., the basic sets coincide with the chain components for $f$. For every 
$C\in\mathcal{C}$, $f|_C$ has the shadowing property, so $C$ (or $(C,f|_C)$) is a transitive SFT.

Consider the case where $f$ is transitive (or, $(X,f)$ is a transitive SFT). Then, we have $X=CR(f)$ and $\mathcal{C}=\mathcal{C}(f)=\{X\}$. Again by \cite[Theorem 3.4.4]{AHi}, $X$ admits a decomposition
\[
X=\bigsqcup_{i=0}^{m-1}f^i(D),
\]
where $m>0$ is a positive integer, such that $f^i(D)$, $0\le i\le m-1$, are clopen $f^m$-invariant subsets of $X$, and
\[
f^m|_{f^i(D)}\colon f^i(D)\to f^i(D)
\]
is mixing for every $0\le i\le m-1$. Here, a continuous map $g\colon Y\to Y$ is said to be {\em mixing} if for any two non-empty open subsets $U$, $V$ of $Y$, there is $N>0$ such that $g^n(U)\cap V\ne\emptyset$ for all $n\ge N$. In this case, we easily see that $\mathcal{D}(f)=\{f^i(D)\colon0\le i\le m-1\}$.

\section{Preparatory lemmas}

In this section, we prove some preparatory lemmas needed for the proof of main results. The first two lemmas give an expression of the chain recurrent set (resp. chain components) for the inverse limit system. 

\begin{lem}
Let $\pi=(\pi_n^{n+1}\colon(X_{n+1},f_{n+1})\to(X_n,f_n))_{n\ge1}$ be an inverse sequence of equivariant maps and let $(X,f)=\lim_\pi(X_n,f_n)$. Then,
\[
CR(f)=\{x=(x_n)_{n\ge1}\in X\colon x_n\in CR(f_n),\forall n\ge 1\}.
\]
\end{lem}

\begin{proof}
Let $R$ denote  the right-hand side of the equation. $CR(f)\subset R$ is clearly true.
Let us prove $R\subset CR(f)$. Note that $\pi_n^{n+1}(CR(f_{n+1}))\subset CR(f_n)$ for every $n\ge1$. Let $Y_n=CR(f_n)$, $g_n=(f_n)|_{Y_n}\colon Y_n\to Y_n$, and $\tilde{\pi}_n^{n+1}=(\pi_n^{n+1})|_{Y_{n+1}}\colon Y_{n+1}\to Y_n$ for each $n\ge1$. Consider the inverse sequence of equivariant maps
\[
\tilde{\pi}=(\tilde{\pi}_n^{n+1}\colon(Y_{n+1},g_{n+1})\to(Y_n,g_n))_{n\ge1}
\]
and let $(Y,g)=\lim_{\tilde{\pi}}(Y_n,g_n)$. Since $g_n$ is chain recurrent for all $n\ge1$, by Lemma 2.3, $g$ is chain recurrent. On the other hand, $R$ is a closed $f$-invariant subset of $X$ and satisfies $Y=R$. The inclusion $i\colon Y\to R$ gives a topological conjugacy $i\colon(Y,g)\to(R,f|_R)$, and so $f|_R\colon R\to R$ is chain recurrent, which clearly implies $R\subset CR(f)$; therefore, the lemma has been proved.
\end{proof}

Let $\pi=(\pi_n^{n+1}\colon(X_{n+1},f_{n+1})\to(X_n,f_n))_{n\ge1}$ be an inverse sequence of equivariant maps and let $(X,f)=\lim_\pi(X_n,f_n)$. Note that for any $n\ge1$ and $C_{n+1}\in\mathcal{C}(f_{n+1})$, there is $C_n\in\mathcal{C}(f_n)$ such that $\pi_n^{n+1}(C_{n+1})\subset C_n$. Let
\[
\mathcal{C}_\pi=\{C_\ast=(C_n)_{n\ge1}\in\prod_{n\ge1}\mathcal{C}(f_n)\colon \pi_n^{n+1}(C_{n+1})\subset C_n,\forall n\ge1\}.
\]
Also, for any $C_\ast=(C_n)_{n\ge1}\in\mathcal{C}_\pi$, let
\[
[C_\ast]=\{x=(x_n)_{n\ge1}\in X\colon x_n\in C_n,\forall n\ge1\}.
\]

\begin{lem}
$\mathcal{C}(f)=\{[C_\ast]\colon C_\ast\in\mathcal{C}_\pi\}$.
\end{lem}

\begin{proof}
For any $C_\ast=(C_n)_{n\ge1}\in\mathcal{C}_\pi$, $[C_\ast]$ is a closed $f$-invariant subset of $X$. We prove that $f|_{[C_\ast]}\colon[C_\ast]\to[C_\ast]$ is chain transitive. For each $n\ge1$, let $g_n=f|_{C_n}\colon C_n\to C_n$ and let
\[
\tilde{\pi}_n^{n+1}=(\pi_n^{n+1})|_{C_{n+1}}\colon C_{n+1}\to C_n.
\]
Consider the inverse sequence of equivariant maps
\[
\tilde{\pi}=(\tilde{\pi}_n^{n+1}\colon(C_{n+1},g_{n+1})\to(C_n,g_n))_{n\ge1}
\]
and let $(Y,g)=\lim_{\tilde{\pi}}(C_n,g_n)$. Since $g_n$ is chain transitive for all $n\ge1$, by Lemma 2.3, $g$ is chain transitive. On the other hand, we have $Y=[C_\ast]$. The inclusion $i\colon Y\to[C_\ast]$ gives a topological conjugacy $i\colon(Y,g)\to([C_\ast],f|_{[C_\ast]})$, which implies that $f|_{[C_\ast]}$ is chain transitive.

Given any $C_\ast=(C_n)_{n\ge1}\in\mathcal{C}_\pi$, from what is shown above, there is $C\in\mathcal{C}(f)$ such that $[C_\ast]\subset C$. Fix $x=(x_n)_{n\ge1}\in[C_\ast]$. Then, for every $y=(y_n)_{n\ge1}\in C$, we easily see that $\{x_n,y_n\}\in C_n$ for all $n\ge1$; therefore, $y\in[C_\ast]$. This implies $C\subset[C_\ast]$ and so $[C_\ast]=C$,
proving
\[
\{[C_\ast]\colon C_\ast\in\mathcal{C}_\pi\}\subset\mathcal{C}(f).
\]
To prove
\[
\mathcal{C}(f)\subset\{[C_\ast]\colon C_\ast\in\mathcal{C}_\pi\},
\]
for any $C\in\mathcal{C}(f)$, fix $x=(x_n)_{n\ge1}\in C$, and take $C_n\in\mathcal{C}(f_n)$ with $x_n\in C_n$ for each $n\ge1$. Then, $C_\ast=(C_n)_{n\ge1}\in\mathcal{C}_\pi$ and $x\in[C_\ast]\subset C$. Similarly as above, we obtain $C=[C_\ast]$, completing the proof.
\end{proof}

The next lemma gives an expression of $\mathcal{D}(f)$, which is introduced in Section 2.2, for the inverse limit system under MLC(1). Given
\[
\pi=(\pi_n^{n+1}\colon(X_{n+1},f_{n+1})\to(X_n,f_n))_{n\ge1},
\]
an inverse sequence of equivariant maps, let $(X,f)=\lim_\pi(X_n,f_n)$ and suppose that $f_n\colon X_n\to X_n$ is chain transitive for all $n\ge1$. Then, by Lemma 2.3, $f\colon X\to X$ is chain transitive. Note that for any $n\ge1$ and $D_{n+1}\in\mathcal{D}(f_{n+1})$, there is $D_n\in\mathcal{D}(f_n)$ such that $\pi_n^{n+1}(D_{n+1})\subset D_n$. Let
\[
\mathcal{D}_\pi=\{D_\ast=(D_n)_{n\ge1}\in\prod_{n\ge1}\mathcal{D}(f_n)\colon \pi_n^{n+1}(D_{n+1})\subset D_n,\forall n\ge1\}.
\]
Also, for any $D_\ast=(D_n)_{n\ge1}\in\mathcal{D}_\pi$, let
\[
[D_\ast]=\{x=(x_n)_{n\ge1}\in X\colon x_n\in D_n,\forall n\ge1\}.
\]

\begin{lem}
If $\pi$ satisfies MLC(1), then $\mathcal{D}(f)=\{[D_\ast]\colon D_\ast\in\mathcal{D}_\pi\}$.
\end{lem}

\begin{proof}
Let $D_\ast=(D_n)_{n\ge1}\in\mathcal{D}_\pi$ and let $x=(x_n)_{n\ge1},y=(y_n)_{n\ge1}\in[D_\ast]$. We prove that $(x,y)\in X^2$ is chain proximal for $f$. Fix any $N>0$ and $\delta>0$. Since $\{x_{N+1},y_{N+1}\}\subset D_{N+1}\in\mathcal{D}(f_{N+1})$, $(x_{N+1},y_{N+1})\in X_{N+1}^2$ is chain proximal for $f_{N+1}$, implying that there is a pair
\[
((x_{N+1}^{(i)})_{i=0}^{k},(y_{N+1}^{(i)})_{i=0}^{k})
\]
of $\delta$-chains of $f_{N+1}$ such that $(x_{N+1}^{(0)},y_{N+1}^{(0)})=(x_{N+1},y_{N+1})$ and $x_{N+1}^{(k)}=y_{N+1}^{(k)}$. Let $(z^{(0)},w^{(0)})=(x,y)\in X^2$ and note that
\[
(z_n^{(0)},w_n^{(0)})=(x_n,y_n)=(\pi_n^{N+1}(x_{N+1}),\pi_n^{N+1}(y_{N+1}))=(\pi_n^{N+1}(x_{N+1}^{(0)}),\pi_n^{N+1}(y_{N+1}^{(0)}))
\]
for every $1\le n\le N$. For each $0<i\le k$, since
\[
\{\pi_N^{N+1}(x_{N+1}^{(i)}),\pi_N^{N+1}(y_{N+1}^{(i)})\}\subset\pi_N^{N+1}(X_{N+1})=\hat{X}_N\quad
\]
by MLC(1) (see Lemma 2.1), there are $(z^{(i)},w^{(i)})\in X^2$, $0<i \le k-1$, and $z^{(k)}=w^{(k)}\in X$ such that
\[
(z_n^{(i)},w_n^{(i)})=(\pi_n^{N+1}(x_{N+1}^{(i)}),\pi_n^{N+1}(y_{N+1}^{(i)}))
\]
and also
\[
z_n^{(k)}=w_n^{(k)}=\pi_n^{N+1}(x_{N+1}^{(k)})=\pi_n^{N+1}(y_{N+1}^{(k)})
\]
for every $1\le n\le N$. Let $d_n$, $n\ge1$, be the metric on $X_n$. For any $0\le i\le k-1$ and $1\le n\le N$, we have
\begin{equation*}
\begin{aligned}
d_n(f(z^{(i)})_n,z_n^{(i+1)})&=d_n(f_n(z_n^{(i)}),z_n^{(i+1)})\\
&=d_n(f_n(\pi_n^{N+1}(x_{N+1}^{(i)})),\pi_n^{N+1}(x_{N+1}^{(i+1)}))\\
&=d_n(\pi_n^{N+1}(f_{N+1}(x_{N+1}^{(i)})),\pi_n^{N+1}(x_{N+1}^{(i+1)}))
\end{aligned}
\end{equation*}
with $d_{N+1}(f_{N+1}(x_{N+1}^{(i)}),x_{N+1}^{(i+1)})\le\delta$, and similarly,
\[
d_n(f(w^{(i)})_n,w_n^{(i+1)})=d_n(\pi_n^{N+1}(f_{N+1}(y_{N+1}^{(i)})),\pi_n^{N+1}(y_{N+1}^{(i+1)}))
\]
with $d_{N+1}(f_{N+1}(y_{N+1}^{(i)}),y_{N+1}^{(i+1)})\le\delta$. Therefore, for every $\epsilon>0$, if $N$ is large enough, and then $\delta$ is sufficiently small,
\[
((z^{(i)})_{i=0}^{k},(w^{(i)})_{i=0}^{k})
\]
is a pair of $\epsilon$-chains of $f$ with $(z^{(0)},w^{(0)})=(x,y)$ and $z^{(k)}=w^{(k)}$, proving that $(x,y)\in X^2$ is chain proximal for $f$. 

Given any $D_\ast=(D_n)_{n\ge1}\in\mathcal{D}_\pi$, from what is shown above, we have  $[D_\ast]\subset D$ for some $D\in\mathcal{D}(f)$. The rest of the proof is identical to that of Lemma 3.2.
\end{proof}

The final lemma gives a sufficient condition to consider the inverse sequence of subsystems without losing MLC(1).

\begin{lem}
Let $\pi=(\pi_n^{n+1}\colon(X_{n+1},f_{n+1})\to(X_n,f_n))_{n\ge1}$ be an inverse sequence of equivariant maps with MLC(1). Let $(X,f)=\lim_\pi(X_n,f_n)$ and suppose that a sequence of closed $f_n$-invariant subsets $Y_n$ of $X_n$, $n\ge1$, has the following properties:
\begin{itemize}
\item[(1)] $\pi_n^{n+1}(Y_{n+1})\subset Y_n$ for every $n\ge1$,
\item[(2)] Any $x=(x_n)_{n\ge1}\in X$ satisfies $x_n\in Y_n$ for all $n\ge1$.
\end{itemize}
For each $n\ge1$, let $g_n=(f_n)|_{Y_n}\colon Y_n\to Y_n$ and let $\tilde{\pi}_n^{n+1}=(\pi_n^{n+1})|_{Y_{n+1}}\colon Y_{n+1}\to Y_n$. Then, the inverse sequence of equivariant maps
\[
\tilde{\pi}=(\tilde{\pi}_n^{n+1}\colon(Y_{n+1},g_{n+1})\to(Y_n,g_n))_{n\ge1}
\]
satisfies MLC(1).
\end{lem}

\begin{proof}
For any $n\ge1$ and $q\in X_{n+1}$, since $\pi_n^{n+1}(q)\in\pi_n^{n+1}(X_{n+1})=\hat{X}_n$ by MLC(1) of $\pi$ (see Lemma 2.1), we have $x_n=\pi_n^{n+1}(q)$ for some $x=(x_n)_{n\ge1}\in X$. Since $\pi_n^{n+1}(q)=x_n=\pi_n^{n+2}(x_{n+2})$, by the property (2), we obtain $\pi_n^{n+1}(q)\in\pi_n^{n+2}(Y_{n+2})$, implying $\pi_n^{n+1}(X_{n+1})\subset\pi_n^{n+2}(Y_{n+2})$. Then,
\[
\pi_n^{n+1}(Y_{n+1})\subset\pi_n^{n+1}(X_{n+1})\subset\pi_n^{n+2}(Y_{n+2})\subset\pi_n^{n+1}(Y_{n+1}),
\]
therefore, $\pi_n^{n+1}(Y_{n+1})=\pi_n^{n+2}(Y_{n+2})$. Since $n\ge1$ is arbitrary, $\tilde{\pi}$ satisfies MLC(1).
\end{proof}

\section{Reduction of Theorem 1.2 to Lemma 1.3}

The aim of this section is to prove the following lemma to reduce Theorem 1.2 to Lemma 1.3.

\begin{lem}
Let $f\colon X\to X$ be a continuous map with the shadowing property. If $\dim X=0$ and $h_{top}(f)>0$, then there is $C\in\mathcal{C}(f)$ such that $f|_{C}\colon C\to C$ has the shadowing property and satisfies $h_{top}(f|_C)>0$.
\end{lem}

A lemma is needed for the proof. It states that for an inverse sequence of SFTs, we can consider the inverse sequence of chain recurrent sets without losing MLC(1).
 
\begin{lem}
Let $\pi=(\pi_n^{n+1}\colon(X_{n+1},f_{n+1})\to(X_n,f_n))_{n\ge1}$ be an inverse sequence of equivariant maps with MLC(1). Let $(X,f)=\lim_{\pi}(X_n,f_n)$ and suppose that $(X_n,f_n)$ is a SFT for each $n\ge1$. Let $Y_n=CR(f_n)$, $g_n=(f_n)|_{Y_n}\colon Y_n\to Y_n$, and $\tilde{\pi}_n^{n+1}=(\pi_n^{n+1})|_{Y_{n+1}}\colon Y_{n+1}\to Y_n$ for every $n\ge1$. Then, the inverse sequence of equivariant maps
\[
\tilde{\pi}=(\tilde{\pi}_n^{n+1}\colon(Y_{n+1},g_{n+1})\to(Y_n,g_n))_{n\ge1}
\]
satisfies MLC(1).
\end{lem}

\begin{proof}
Let $n\ge1$. Since $Y_{n+1}=\overline{Per(f_{n+1})}$ and $\pi_n^{n+1}(Per(f_{n+1}))\subset Per(f_n)$,
\[
\pi_n^{n+1}(Y_{n+1})=\pi_n^{n+1}(\overline{Per(f_{n+1})})\subset\overline{\pi_n^{n+1}(Per(f_{n+1}))}\subset\overline{\pi_n^{n+1}(Y_{n+1})\cap Per(f_n)}.
\]
By MLC(1) of $\pi$, $\pi_n^{n+1}(Y_{n+1})\subset\pi_n^{n+1}(X_{n+1})=\pi_n^{n+2}(X_{n+2})$; therefore, for any $p\in\pi_n^{n+1}(Y_{n+1})\cap Per(f_n)$, there is $q\in X_{n+2}$ such that $p=\pi_n^{n+2}(q)$. Then, there is $r\in Y_{n+2}$ such that
\[
\lim_{k\to\infty}d_{n+2}(f_{n+2}^k(q),f_{n+2}^k(r))=0,
\]
implying 
\begin{equation*}
\begin{aligned}
\lim_{k\to\infty}d_n(f_n^k(p),f_n^k(\pi_{n}^{n+2}(r)))&=\lim_{k\to\infty}d_n(f_n^k(\pi_{n}^{n+2}(q)),f_n^k(\pi_{n}^{n+2}(r)))\\
&=\lim_{k\to\infty}d_n(\pi_{n}^{n+2}(f_{n+2}^k(q)),\pi_{n}^{n+2}(f_{n+2}^k(r)))\\
&=0,
\end{aligned}
\end{equation*}
where $d_n$, $d_{n+2}$ are the metrics on $X_n$, $X_{n+2}$. Note that $\pi_{n}^{n+2}(r)\in\pi_n^{n+2}(Y_{n+2})$. From $p\in Per(f_n)$ and the $f_n$-invariance of $\pi_n^{n+2}(Y_{n+2})$, it follows that $p\in\overline{\pi_n^{n+2}(Y_{n+2})}=\pi_n^{n+2}(Y_{n+2})$. Since $p\in\pi_n^{n+1}(Y_{n+1})\cap Per(f_n)$ is arbitrary, we obtain
\[
\pi_n^{n+1}(Y_{n+1})\cap Per(f_n)\subset\pi_n^{n+2}(Y_{n+2})
\]
and so
\[
\pi_n^{n+1}(Y_{n+1})\subset\overline{\pi_n^{n+1}(Y_{n+1})\cap Per(f_n)}\subset\overline{\pi_n^{n+2}(Y_{n+2})}=\pi_n^{n+2}(Y_{n+2})
\]
Thus, $\pi_n^{n+1}(Y_{n+1})=\pi_n^{n+2}(Y_{n+2})$, proving the lemma.
\end{proof}

Then, we prove Lemma 4.1. The proof is based on Lemma 1.2 and by carefully choosing an inverse sequence of chain components with MLC(1).

\begin{proof}[Proof of Lemma 4.1]
By Lemma 1.2 and Lemma 2.4, we may assume $(X,f)=\lim_{\pi}(X_n,f_n)$, where $(X_n,f_n)$, $n\ge1$, are SFTs, and
\[
\pi=(\pi_n^{n+1}\colon(X_{n+1},f_{n+1})\to(X_n,f_n))_{n\ge1}
\]
is an inverse sequence of equivariant maps with MLC(1). Let $Y_n=CR(f_n)$, $g_n=(f_n)|_{Y_n}\colon Y_n\to Y_n$, and $\tilde{\pi}_n^{n+1}=(\pi_n^{n+1})|_{Y_{n+1}}\colon Y_{n+1}\to Y_n$ for every $n\ge1$. Then, $(Y_n,g_n)$, $n\ge1$, are chain recurrent SFTs, and by Lemma 4.2,
\[
\tilde{\pi}=(\tilde{\pi}_n^{n+1}\colon(Y_{n+1},g_{n+1})\to(Y_n,g_n))_{n\ge1}
\]
satisfies MLC(1). Letting $(Y,g)=\lim_{\tilde{\pi}}(Y_n,g_n)$ and
\[
R=\{x=(x_n)_{n\ge1}\in X\colon x_n\in CR(f_n),\forall n\ge 1\},
\]
we have $R=CR(f)$ by Lemma 3.1, and as in the proof of Lemma 3.1, the inclusion $i\colon Y\to CR(f)$ is a topological conjugacy $i\colon(Y,g)\to(CR(f),f|_{CR(f)})$. By this, again without loss of generality, we may assume that $f$ and $f_n$, $n\ge1$, are chain recurrent.

Since $h_{top}(f)>0$, there is $C^\dagger\in\mathcal{C}(f)$ such that $h_{top}(f|_{C^\dagger})>0$. A proof of this fact is as follows: by the variational principle, there is an ergodic $f$-invariant Borel probability measure $\mu$ on $X$ such that the measure theoretical entropy $h_\mu(f)$ is positive. Since  $f|_{{\rm supp}(\mu)}\colon{\rm supp}(\mu)\to{\rm supp}(\mu)$, the restriction of $f$ to the support of $\mu$, is transitive, there is $C^\dagger\in\mathcal{C}(f)$ such that ${\rm supp}(\mu)\subset C^\dagger$. By the variational principle again, we obtain
\[
h_{top}(f|_{C^\dagger})\ge h_{top}(f|_{{\rm supp}(\mu)})\ge h_\mu(f|_{{\rm supp}(\mu)})=h_\mu(f)>0.
\]
Then, by Lemma 3.2, there is $C_\ast=(C_n)_{n\ge1}\in\mathcal{C}_\pi$ such that $C^\dagger=[C_\ast]$. Letting $\Gamma=\prod_{n\ge1}\pi_n^{n+1}(C_{n+1})$, a closed $f$-invariant subset of $X$, since $C^\dagger\subset\Gamma$, we have
\[
0<h_{top}(f|_{C^\dagger})\le h_{top}(f|_\Gamma)=h_{top}(\prod_{n\ge1}(f_n)|_{\pi_n^{n+1}(C_{n+1})})=\sum_{n\ge1}h_{top}((f_n)|_{\pi_n^{n+1}(C_{n+1})}),
\]
implying $h_{top}((f_n)|_{\pi_n^{n+1}(C_{n+1})})>0$ for some $n\ge1$.

Note that $\mathcal{C}(f_m)$, $m\ge1$, are finite sets, and for any $m\ge1$ and $D\in\mathcal{C}(f_m)$, $(f_m)|_{D}\colon D\to D$ is transitive (see Section 2.4). Let us prove the following claim. 
\begin{claim}
\normalfont
There is $C'_\ast=(C'_m)_{m\ge n}\in\prod_{m\ge n}\mathcal{C}(f_m)$ with the following properties:
\begin{itemize}
\item[(1)] $C'_n=C_n$,
\item[(2)] $\pi_n^{n+1}(C_{n+1})\subset\pi_n^{n+1}(C'_{n+1})$,
\item[(3)] $\pi_m^{m+1}(C'_{m+1})\subset C'_m$ for every $m\ge n$,
\item[(4)] $\pi_m^{m+1}(C'_{m+1})=\pi_m^{m+2}(C'_{m+2})$ for all $m\ge n$.
\end{itemize}
\end{claim}

\begin{proof}[proof of the claim]

{\bf Step $1\colon$} Let $C'_n=C_n$. Take $D_{n+1}\in\mathcal{C}(f_{n+1})$ such that 
\[
\pi_n^{n+1}(C_{n+1})\subset\pi_n^{n+1}(D_{n+1})\subset C'_n, \tag{P1}
\]
and $\pi_n^{n+1}(D_{n+1})$ is maximal among
\[
\{\pi_n^{n+1}(E_{n+1})\colon E_{n+1}\in\mathcal{C}(f_{n+1}),\pi_n^{n+1}(C_{n+1})\subset\pi_n^{n+1}(E_{n+1})\subset C'_n\}
\]
with respect to the inclusion relation.

{\bf Step $2\colon$} Note that $(f_n)|_{\pi_n^{n+1}(D_{n+1})}$ is transitive, and take a transitive point $p_1\in\pi_n^{n+1}(D_{n+1})$, i.e., $\pi_n^{n+1}(D_{n+1})=\omega(p_1,f_n)$, the $\omega$-limit set. Since
\[
p_1\in\pi_n^{n+1}(D_{n+1})\subset\pi_n^{n+1}(X_{n+1})=\pi_n^{n+2}(X_{n+2}),
\]
we have $p_1=\pi_n^{n+2}(q_1)$ for some $q_1\in X_{n+2}$. Take $C'_{n+1}\in\mathcal{C}(f_{n+1})$ with $\pi_{n+1}^{n+2}(q_1)\in C'_{n+1}$. Then, choose $D_{n+2}\in\mathcal{C}(f_{n+2})$ such that
\[
\pi_{n+1}^{n+2}(q_1)\in\pi_{n+1}^{n+2}(D_{n+2})\subset C'_{n+1}, \tag{P2}
\]
and $\pi_{n+1}^{n+2}(D_{n+2})$ is maximal among
\[
\{\pi_{n+1}^{n+2}(E_{n+2})\colon E_{n+2}\in\mathcal{C}(f_{n+2}),\pi_{n+1}^{n+2}(q_1)\in\pi_{n+1}^{n+2}(E_{n+2})\subset C'_{n+1}\}
\]
with respect to the inclusion relation. By (P2), we have
\[
p_1\in\pi_n^{n+2}(D_{n+2})\subset\pi_n^{n+1}(C'_{n+1}),
\]
implying
\[
\pi_n^{n+1}(D_{n+1})\subset\pi_n^{n+2}(D_{n+2})\subset\pi_n^{n+1}(C'_{n+1})
\]
since $\pi_n^{n+1}(D_{n+1})=\omega(p_1,f_n)$, and $\pi_n^{n+2}(D_{n+2})$ is $f_n$-invariant. By (P1) and $p_1\in\pi_n^{n+1}(D_{n+1})$, we see that $\pi_n^{n+1}(C_{n+1})\subset\pi_n^{n+1}(D_{n+1})$ and $p_1\in C'_n\cap\pi_n^{n+1}(C'_{n+1})$; therefore,
\[
\pi_n^{n+1}(C_{n+1})\subset\pi_n^{n+1}(D_{n+1})\subset\pi_n^{n+2}(D_{n+2})\subset\pi_n^{n+1}(C'_{n+1})\subset C'_n.
\]
By the maximality of $\pi_n^{n+1}(D_{n+1})$ in Step 1, we obtain
\[
\pi_n^{n+1}(D_{n+1})=\pi_n^{n+2}(D_{n+2})=\pi_n^{n+1}(C'_{n+1}). \tag{Q1}
\]

{\bf Step $3\colon$} Note that $(f_{n+1})|_{\pi_{n+1}^{n+2}(D_{n+2})}$ is transitive, and take a transitive point $p_2\in\pi_{n+1}^{n+2}(D_{n+2})$, i.e., $\pi_{n+1}^{n+2}(D_{n+2})=\omega(p_2,f_{n+1})$. Since
\[
p_2\in\pi_{n+1}^{n+2}(D_{n+2})\subset\pi_{n+1}^{n+2}(X_{n+2})=\pi_{n+1}^{n+3}(X_{n+3}),
\]
we have $p_2=\pi_{n+1}^{n+3}(q_2)$ for some $q_2\in X_{n+3}$. Take $C'_{n+2}\in\mathcal{C}(f_{n+2})$ with $\pi_{n+2}^{n+3}(q_2)\in C'_{n+2}$. Then, choose $D_{n+3}\in\mathcal{C}(f_{n+3})$ such that
\[
\pi_{n+2}^{n+3}(q_2)\in\pi_{n+2}^{n+3}(D_{n+3})\subset C'_{n+2}, \tag{P3}
\]
and $\pi_{n+2}^{n+3}(D_{n+3})$ is maximal among
\[
\{\pi_{n+2}^{n+3}(E_{n+3})\colon E_{n+3}\in\mathcal{C}(f_{n+3}),\pi_{n+2}^{n+3}(q_2)\in\pi_{n+2}^{n+3}(E_{n+3})\subset C'_{n+2}\}
\]
with respect to the inclusion relation. By (P3), we have
\[
p_2\in\pi_{n+1}^{n+3}(D_{n+3})\subset\pi_{n+1}^{n+2}(C'_{n+2}),
\]
implying
\[
\pi_{n+1}^{n+2}(D_{n+2})\subset\pi_{n+1}^{n+3}(D_{n+3})\subset\pi_{n+1}^{n+2}(C'_{n+2})
\]
since $\pi_{n+1}^{n+2}(D_{n+2})=\omega(p_2,f_{n+1})$, and $\pi_{n+1}^{n+3}(D_{n+3})$ is $f_{n+1}$-invariant. By (P2) and $p_2\in\pi_{n+1}^{n+2}(D_{n+2})$, we see that $\pi_{n+1}^{n+2}(q_1)\in\pi_{n+1}^{n+2}(D_{n+2})$ and $p_2\in C'_{n+1}\cap\pi_{n+1}^{n+2}(C'_{n+2})$; therefore,
\[
\pi_{n+1}^{n+2}(q_1)\in\pi_{n+1}^{n+2}(D_{n+2})\subset\pi_{n+1}^{n+3}(D_{n+3})\subset\pi_{n+1}^{n+2}(C'_{n+2})\subset C'_{n+1}.
\]
By the maximality of $\pi_{n+1}^{n+2}(D_{n+2})$ in Step 2, we obtain
\[
\pi_{n+1}^{n+2}(D_{n+2})=\pi_{n+1}^{n+3}(D_{n+3})=\pi_{n+1}^{n+2}(C'_{n+2}). \tag{Q2}
\]
(Q1) and (Q2) yield $\pi_n^{n+1}(C'_{n+1})=\pi_n^{n+2}(C'_{n+2})$.

{\bf Step $4\colon$} Note that $(f_{n+2})|_{\pi_{n+2}^{n+3}(D_{n+3})}$ is transitive, and take a transitive point $p_3\in\pi_{n+2}^{n+3}(D_{n+3})$, i.e., $\pi_{n+2}^{n+3}(D_{n+3})=\omega(p_3,f_{n+2})$. Since
\[
p_3\in\pi_{n+2}^{n+3}(D_{n+3})\subset\pi_{n+2}^{n+3}(X_{n+3})=\pi_{n+2}^{n+4}(X_{n+4}),
\]
we have $p_3=\pi_{n+2}^{n+4}(q_3)$ for some $q_3\in X_{n+4}$. Take $C'_{n+3}\in\mathcal{C}(f_{n+3})$ with $\pi_{n+3}^{n+4}(q_3)\in C'_{n+3}$. Then, choose $D_{n+4}\in\mathcal{C}(f_{n+4})$ such that
\[
\pi_{n+3}^{n+4}(q_3)\in\pi_{n+3}^{n+4}(D_{n+4})\subset C'_{n+3}, \tag{P4}
\]
and $\pi_{n+3}^{n+4}(D_{n+4})$ is maximal among
\[
\{\pi_{n+3}^{n+4}(E_{n+4})\colon E_{n+4}\in\mathcal{C}(f_{n+4}),\pi_{n+3}^{n+4}(q_3)\in\pi_{n+3}^{n+4}(E_{n+4})\subset C'_{n+3}\}
\]
with respect to the inclusion relation. By (P4), we have
\[
p_3\in\pi_{n+2}^{n+4}(D_{n+4})\subset\pi_{n+2}^{n+3}(C'_{n+3}),
\]
implying
\[
\pi_{n+2}^{n+3}(D_{n+3})\subset\pi_{n+2}^{n+4}(D_{n+4})\subset\pi_{n+2}^{n+3}(C'_{n+3})
\]
since $\pi_{n+2}^{n+3}(D_{n+3})=\omega(p_3,f_{n+2})$, and $\pi_{n+2}^{n+4}(D_{n+4})$ is $f_{n+2}$-invariant. By (P3) and $p_3\in\pi_{n+2}^{n+3}(D_{n+3})$, we see that $\pi_{n+2}^{n+3}(q_2)\in\pi_{n+2}^{n+3}(D_{n+3})$ and $p_3\in C'_{n+2}\cap\pi_{n+2}^{n+3}(C'_{n+3})$; therefore,
\[
\pi_{n+2}^{n+3}(q_2)\in\pi_{n+2}^{n+3}(D_{n+3})\subset\pi_{n+2}^{n+4}(D_{n+4})\subset\pi_{n+2}^{n+3}(C'_{n+3})\subset C'_{n+2}.
\]
By the maximality of $\pi_{n+2}^{n+3}(D_{n+3})$ in Step 3, we obtain
\[
\pi_{n+2}^{n+3}(D_{n+3})=\pi_{n+2}^{n+4}(D_{n+4})=\pi_{n+2}^{n+3}(C'_{n+3}). \tag{Q3}
\]
(Q2) and (Q3) yield $\pi_{n+1}^{n+2}(C'_{n+2})=\pi_{n+1}^{n+3}(C'_{n+3})$.

Continuing inductively, we obtain a sequence $C'_\ast=(C'_m)_{m\ge n}\in\prod_{m\ge n}\mathcal{C}(f_m)$. Then, the properties (1) and (2) are ensured in Steps $1$ and $2$. For any $k\ge0$, $\pi_{n+k}^{n+k+1}(C'_{n+k+1})\subset C'_{n+k}$ and $\pi_{n+k}^{n+k+1}(C'_{n+k+1})=\pi_{n+k}^{n+k+2}(C'_{n+k+2})$ are established in Steps $k+2$ and $k+3$, respectively. Thus, $C'_\ast$ satisfies the required properties, and so the claim has been proved.
\end{proof}

We continue the proof. Define $C''_\ast=(C''_j)_{j\ge1}\in\prod_{j\ge1}\mathcal{C}(f_j)$ by
\begin{equation*}
C''_j=
\begin{cases}
C_j&\text{if $1\le j<n$}\\
C'_j&\text{if $n\le j$}
\end{cases}
.
\end{equation*}
By the properties (1) and (3) in the claim, we see $C''_\ast\in\mathcal{C}_\pi$. By Lemma 3.2, letting $C=[C''_\ast]$, we obtain $C\in\mathcal{C}(f)$. Let
\[
\pi''=((\pi_j^{j+1})|_{C''_{j+1}}\colon(C''_{j+1},(f_{j+1})|_{C''_{j+1}})\to(C''_j,(f_j)|_{C''_j}))_{j\ge1}
\]
and
\[
\pi'=((\pi_m^{m+1})|_{C'_{m+1}}\colon(C'_{m+1},(f_{m+1})|_{C'_{m+1}})\to(C'_m,(f_m)|_{C'_m}))_{m\ge n}.
\]
Then, $(C,f|_C)$ (resp.\:$\lim_{\pi''}(C''_j,(f_j)|_{C''_j})$) is topologically conjugate to
\[
\lim_{\pi''}(C''_j,(f_j)|_{C''_j})
\]
(resp.\:$\lim_{\pi'}(C'_m,(f_m)|_{C'_m})$), so $(C,f|_C)$ is topologically conjugate to
\[
\lim_{\pi'}(C'_m,(f_m)|_{C'_m}).
\]

Let $(Y,g)=\lim_{\pi'}(C'_m,(f_m)|_{C'_m})$. By the property (4) in the claim, $\pi'$ satisfies MLC(1). Since $(f_m)|_{C'_m}$ has the shadowing property for each $m\ge1$, due to Lemma 1.1, $g$ has the shadowing property. Let us prove $h_{top}(g)>0$. Again by MLC(1) of $\pi'$, we have $\hat{C}'_n=\pi_n^{n+1}(C'_{n+1})$ (see Lemma 2.1). Then, a map $\phi\colon Y\to\pi_n^{n+1}(C'_{n+1})$ defined by $\phi(y)=y_n$ for all $y=(y_m)_{m\ge n}\in Y$, gives a factor map
\[
\phi\colon(Y,g)\to(\pi_n^{n+1}(C'_{n+1}),(f_n)|_{\pi_n^{n+1}(C'_{n+1})}).
\]
The property (2) in the claim ensures that $\pi_n^{n+1}(C_{n+1})$ is a closed $f_n$-invariant subset of $\pi_n^{n+1}(C'_{n+1})$, therefore,
\[
h_{top}(g)\ge h_{top}((f_n)|_{\pi_n^{n+1}(C'_{n+1})})\ge h_{top}((f_n)|_{\pi_n^{n+1}(C_{n+1})})>0.
\]
Thus, $f|_C$ has the shadowing property and satisfies $h_{top}(f|_C)>0$, completing the proof of the lemma.
\end{proof}

\section{Proof of Lemma 1.3}

In this section, we prove Lemma 1.3. Let
\[
\pi=(\pi_n^{n+1}\colon(X_{n+1},f_{n+1})\to(X_n,f_n))_{n\ge1}
\]
be an inverse sequence of equivariant maps with MLC(1) and suppose that $(X_n,f_n)$ is a transitive SFT for each $n\ge1$. For every $n\ge1$, note that $\mathcal{D}(f_n)$ is a finite set, and let $m_n=|\mathcal{D}(f_n)|$. Then, $m_n|m_{n+1}$ for all $n\ge1$, and for any $n\ge1$ and $E\in\mathcal{D}(f_n)$, $(f_n^{m_n})|_E\colon E\to E$ is mixing. The proof of the first lemma is similar to that of Lemma 4.1.

\begin{lem}
There is $D_\ast=(D_n)_{n\ge1}\in\mathcal{D}_\pi$ such that
\[
\tilde{\pi}=((\pi_n^{n+1})|_{D_{n+1}}\colon D_{n+1}\to D_n)_{n\ge1}
\]
satisfies MLC(1).
\end{lem}
 
\begin{proof}

We argue as in the proof of Lemma 4.1.

{\bf Step $1\colon$} Fix $D_1\in\mathcal{D}(f_1)$ with $\pi_1^2(F_2)\subset D_1$ for some $F_2\in\mathcal{D}(f_2)$. Take $E_2\in\mathcal{D}(f_2)$ such that 
\[
\pi_1^2(E_2)\subset D_1, \tag{P1}
\]
and $\pi_1^2(E_2)$ is maximal among
\[
\{\pi_1^2(F_2)\colon F_2\in\mathcal{D}(f_2),\pi_1^2(F_2)\subset D_1\}
\]
with respect to the inclusion relation.

{\bf Step $2\colon$} Note that $(f_1^{m_2})|_{\pi_1^2(E_2)}$ is mixing, so $(f_1^{m_3})|_{\pi_1^2(E_2)}$ is transitive, and take a transitive point $p_1\in\pi_1^2(E_2)$, i.e., $\pi_1^2(E_2)=\omega(p_1,f_1^{m_3})$, the $\omega$-limit set. Since
\[
p_1\in\pi_1^2(E_2)\subset\pi_1^2(X_2)=\pi_1^3(X_3),
\]
we have $p_1=\pi_1^3(q_1)$ for some $q_1\in X_3$. Take $D_2\in\mathcal{D}(f_2)$ with $\pi_2^3(q_1)\in D_2$. Then, choose $E_3\in\mathcal{D}(f_3)$ such that
\[
\pi_2^3(q_1)\in\pi_2^3(E_3)\subset D_2, \tag{P2}
\]
and $\pi_2^3(E_3)$ is maximal among
\[
\{\pi_2^3(F_3)\colon F_3\in\mathcal{D}(f_3),\pi_2^3(q_1)\in\pi_2^3(F_3)\subset D_2\}
\]
with respect to the inclusion relation. By (P2), we have
\[
p_1\in\pi_1^3(E_3)\subset\pi_1^2(D_2),
\]
implying
\[
\pi_1^2(E_2)\subset\pi_1^3(E_3)\subset\pi_1^2(D_2)
\]
since $\pi_1^2(E_2)=\omega(p_1,f_1^{m_3})$, and $\pi_1^3(E_3)$ is $f_1^{m_3}$-invariant. By (P1) and $p_1\in\pi_1^2(E_2)$, we see that $p_1\in D_1\cap\pi_1^2(D_2)$; therefore,
\[
\pi_1^2(E_2)\subset\pi_1^3(E_3)\subset\pi_1^2(D_2)\subset D_1.
\]
By the maximality of $\pi_1^2(E_2)$ in Step 1, we obtain
\[
\pi_1^2(E_2)=\pi_1^3(E_3)=\pi_1^2(D_2). \tag{Q1}
\]

{\bf Step $3\colon$} Note that $(f_2^{m_3})|_{\pi_2^3(E_3)}$ is mixing, so $(f_2^{m_4})|_{\pi_2^3(E_3)}$ is transitive, and take a transitive point $p_2\in\pi_2^3(E_3)$, i.e., $\pi_2^3(E_3)=\omega(p_2,f_2^{m_4})$. Since
\[
p_2\in\pi_2^3(E_3)\subset\pi_2^3(X_3)=\pi_2^4(X_4),
\]
we have $p_2=\pi_2^4(q_2)$ for some $q_2\in X_4$. Take $D_3\in\mathcal{D}(f_3)$ with $\pi_3^4(q_2)\in D_3$. Then, choose $E_4\in\mathcal{D}(f_4)$ such that
\[
\pi_3^4(q_2)\in\pi_3^4(E_4)\subset D_3, \tag{P3}
\]
and $\pi_3^4(E_4)$ is maximal among
\[
\{\pi_3^4(F_4)\colon F_4\in\mathcal{D}(f_4),\pi_3^4(q_2)\in\pi_3^4(F_4)\subset D_3\}
\]
with respect to the inclusion relation. By (P3), we have
\[
p_2\in\pi_2^4(E_4)\subset\pi_2^3(D_3),
\]
implying
\[
\pi_2^3(E_3)\subset\pi_2^4(E_4)\subset\pi_2^3(D_3)
\]
since $\pi_2^3(E_3)=\omega(p_2,f_2^{m_4})$, and $\pi_2^4(E_4)$ is $f_2^{m_4}$-invariant. By (P2) and $p_2\in\pi_2^3(E_3)$, we see that $\pi_2^3(q_1)\in\pi_2^3(E_3)$ and $p_2\in D_2\cap\pi_2^3(D_3)$; therefore,
\[
\pi_2^3(q_1)\in\pi_2^3(E_3)\subset\pi_2^4(E_4)\subset\pi_2^3(D_3)\subset D_2.
\]
By the maximality of $\pi_2^3(E_3)$ in Step 2, we obtain
\[
\pi_2^3(E_3)=\pi_2^4(E_4)=\pi_2^3(D_3). \tag{Q2}
\]
(Q1) and (Q2) yield $\pi_1^2(D_2)=\pi_1^3(D_3)$.

{\bf Step $4\colon$} Note that $(f_3^{m_4})|_{\pi_3^4(E_4)}$ is mixing, so $(f_3^{m_5})|_{\pi_3^4(E_4)}$ is transitive, and take a transitive point $p_3\in\pi_3^4(E_4)$, i.e., $\pi_3^4(E_4)=\omega(p_3,f_3^{m_5})$. Since
\[
p_3\in\pi_3^4(E_4)\subset\pi_3^4(X_4)=\pi_3^5(X_5),
\]
we have $p_3=\pi_3^5(q_3)$ for some $q_3\in X_5$. Take $D_4\in\mathcal{D}(f_4)$ with $\pi_4^5(q_3)\in D_4$. Then, choose $E_5\in\mathcal{D}(f_5)$ such that
\[
\pi_4^5(q_3)\in\pi_4^5(E_5)\subset D_4, \tag{P4}
\]
and $\pi_4^5(E_5)$ is maximal among
\[
\{\pi_4^5(F_5)\colon F_5\in\mathcal{D}(f_5),\pi_4^5(q_3)\in\pi_4^5(F_5)\subset D_4\}
\]
with respect to the inclusion relation. By (P4), we have
\[
p_3\in\pi_3^5(E_5)\subset\pi_3^4(D_4),
\]
implying
\[
\pi_3^4(E_4)\subset\pi_3^5(E_5)\subset\pi_3^4(D_4)
\]
since $\pi_3^4(E_4)=\omega(p_3,f_3^{m_5})$, and $\pi_3^5(E_5)$ is $f_3^{m_5}$-invariant. By (P3) and $p_3\in\pi_3^4(E^4)$, we see that $\pi_3^4(q_2)\in\pi_3^4(E_4)$ and $p_3\in D_3\cap\pi_3^4(D_4)$; therefore,
\[
\pi_3^4(q_2)\in\pi_3^4(E_4)\subset\pi_3^5(E_5)\subset\pi_3^4(D_4)\subset D_3.
\]
By the maximality of $\pi_3^4(E_4)$ in Step 3, we obtain
\[
\pi_3^4(E_4)=\pi_3^5(E_5)=\pi_3^4(D_4). \tag{Q3}
\]
(Q2) and (Q3) yield $\pi_2^3(D_3)=\pi_2^4(D_4)$.

Continuing inductively, we obtain a sequence $D_\ast=(D_n)_{n\ge1}\in\prod_{n\ge1}\mathcal{D}(f_n)$. For any $n\ge1$, $\pi_n^{n+1}(D_{n+1})\subset D_n$ and $\pi_n^{n+1}(D_{n+1})=\pi_n^{n+2}(D_{n+2})$ are established in Steps $n+1$ and $n+2$, respectively. Thus, $D_\ast\in\mathcal{D}_\pi$, and 
\[
\tilde{\pi}=((\pi_n^{n+1})|_{D_{n+1}}\colon D_{n+1}\to D_n)_{n\ge1}
\]
satisfies MLC(1), completing the proof.
\end{proof}

The next lemma relates the previous lemma to a method developed in \cite{Ka2}. Let $\pi=(\pi_n^{n+1}\colon(X_{n+1},f_{n+1})\to(X_n,f_n))_{n\ge1}$ be an inverse sequence of equivariant maps with MLC(1) and let $(X,f)=\lim_\pi(X_n,f_n)$. Suppose that $(X_n,f_n)$, $n\ge1$, are transitive SFTs, and for $D_\ast=(D_n)_{n\ge1}\in\mathcal{D}_\pi$,
\[
\tilde{\pi}=((\pi_n^{n+1})|_{D_{n+1}}\colon D_{n+1}\to D_n)_{n\ge1}
\]
satisfies MLC(1). By Lemma 3.3, letting $D=[D_\ast]$, we have $D\in\mathcal{D}(f)$. Let $d$, $d_n$, $n\ge1$, be the metrics on $X$, $X_n$.

\begin{lem}
For any $\epsilon>0$, there is $\delta>0$ such that every $\delta$-pseudo orbit $(x^{(i)})_{i\ge0}$ of $f$ with $x^{(0)}\in D$ is $\epsilon$-shadowed by some $x\in D$.
\end{lem}

\begin{proof}
Fix any $N>0$ and $\epsilon'>0$. Note that $f_{N+1}\colon X_{N+1}\to X_{N+1}$ has the shadowing property, and $\mathcal{D}(f_{N+1})$ is a clopen partition of $X_{N+1}$; therefore, there is $\delta'>0$ such that for any $E_{N+1}\in\mathcal{D}_{N+1}$, every $\delta'$-pseudo orbit $(y_{N+1}^{(i)})_{i\ge0}$
of $f_{N+1}$ with $y_{N+1}^{(0)}\in E_{N+1}$ is $\epsilon'$-shadowed by some $y_{N+1}\in E_{N+1}$.

If $\delta>0$ is small enough, then for every $\delta$-pseudo orbit $\xi=(x^{(i)})_{i\ge0}$ of $f$ with $x^{(0)}\in D$, $\xi_{N+1}=(x_{N+1}^{(i)})_{i\ge0}$ is a $\delta'$-pseudo orbit of $f_{N+1}$ with $x_{N+1}^{(0)}\in D_{N+1}$, which is $\epsilon'$-shadowed by some $x_{N+1}\in D_{N+1}$. Since
\[
\pi_N^{N+1}(x_{N+1})\in\pi_N^{N+1}(D_{N+1})=\hat{D}_N
\]
by MLC(1) of $\tilde{\pi}$ (see Lemma 2.1), there is $x=(x_n)_{n\ge1}\in D=[D_\ast]$ such that $x_n=\pi_n^{N+1}(x_{N+1})$ for each $1\le n\le N$. Then, for any $i\ge0$ and $1\le n\le N$, we have
\begin{equation*}
\begin{aligned}
d_n(f^i(x)_n,x_n^{(i)})&=d_n(f_n^i(x_n),x_n^{(i)})\\
&=d_n(f_n^i(\pi_n^{N+1}(x_{N+1})),\pi_n^{N+1}(x_{N+1}^{(i)}))\\
&=d_n(\pi_n^{N+1}(f_{N+1}^i(x_{N+1})),\pi_n^{N+1}(x_{N+1}^{(i)}))
\end{aligned}
\end{equation*}
with $d_{N+1}(f_{N+1}^i(x_{N+1}),x_{N+1}^{(i)})\le\epsilon'$. Therefore, for every $\epsilon>0$, if $N$ is large enough, and then $\epsilon'$ is sufficiently small, we have $d(f^i(x),x^{(i)})\le\epsilon$ for all $i\ge0$, i.e., $\xi$ is $\epsilon$-shadowed by $x\in D$. Since $\xi$ is arbitrary, the lemma has been proved. 
\end{proof}

To prove Lemma 1.3, we use the method in \cite{Ka2}. The next lemma is a modification of \cite[Lemma 2.6]{Ka2}.

\begin{lem}
Let $f\colon X\to X$ be a chain transitive continuous map and let $D\in\mathcal{D}(f)$. Suppose that for any $\epsilon>0$, there is $\delta>0$ such that every $\delta$-pseudo orbit $(x_i)_{i\ge0}$ of $f$ with $x_0\in D$ is $\epsilon$-shadowed by some $x\in D$. Then, for any $y,z\in D$ and $\epsilon>0$, there is $w\in D$ such that $d(z,w)\le\epsilon$ and $\limsup_{k\to\infty}d(f^k(y),f^k(w))\le\epsilon$. 
\end{lem}

\begin{proof}
Given any $\epsilon>0$, take $\delta>0$ as in the assumption. For this $\delta$, choose $N>0$ as in the property (3) of $\sim_{f,\delta}$ (see Section 2.2.3). Note that $y,z\in D$ implies $y\sim_f z$ and so $y\sim_{f,\delta}z$. Since $y\sim_{f,\delta}f^{mN}(y)$, we have $z\sim_{f,\delta}f^{mN}(y)$. Then, the choice of $N$ gives a $\delta$-chain $\alpha=(y_i)_{i=0}^{mN}$ of $f$ with $y_0=z$ and $y_{mN}=f^{mN}(y)$. Let
\[
\beta=(f^{mN}(y),f^{mN+1}(y),\dots)
\]
and $\xi=\alpha\beta=(x_i)_{i\ge0}$. Then, $\xi$ is a $\delta$-pseudo orbit of $f$ with $x_0=z\in D$, so is $\epsilon$-shadowed by some $w\in D$. Note that $d(z,w)=d(x_0,w)\le\epsilon$. Also, we have
\[
d(f^i(y),f^i(w))=d(x_i,f^i(w))\le\epsilon
\]
for every $i\ge mN$, so $\limsup_{k\to\infty}d(f^k(y),f^k(w))\le\epsilon$. This completes the proof.
\end{proof}

Let $f\colon X\to X$ be a continuous map. For $n\ge 2$ and $r>0$, we say that an $n$-tuple $(x_1,x_2,\dots,x_n)\in X^n$ is {\em $r$-distal} if
\[
\inf_{k\ge0}\min_{1\le i<j\le n}d(f^k(x_i),f^k(x_j))\ge r.
\]
Then, the following lemma is a consequence of Lemma 2.4 and Lemma 2.5 in \cite{Ka2}.

\begin{lem}
Suppose that a continuous map $f\colon X\to X$ is chain transitive and has the shadowing property. If $h_{top}(f)>0$, then for any $n\ge2$, there is $r_n>0$ such that for every $D\in\mathcal{D}(f)$, there is an $r_n$-distal $n$-tuple $(x_1,x_2,\dots,x_n)\in X^n$ with $\{x_1,x_2,\dots,x_n\}\subset D$.
\end{lem}

We recall a simplified version of Mycielski's theorem \cite[Theorem 1]{My}.  A topological space is said to be {\em perfect} if it has no isolated point.

\begin{lem}
Let $X$ be a perfect complete metric space. If $R_n$ is a residual subset of $X^n$ for each $n\ge2$, then there is a Mycielski set $S$ which is dense in $X$ and satisfies
$(x_1,x_2,\dots,x_n)\in R_n$ for any $n\ge2$ and distinct $x_1,x_2,\dots,x_n\in S$.
\end{lem}

Finally, we complete the proof of Lemma 1.3. 

\begin{proof}[Proof of Lemma 1.3]
By Lemma 1.2 and Lemma 2.4, we may assume $(X,f)=\lim_{\pi}(X_j,f_j)$, where $(X_j,f_j)$, $j\ge1$, are SFTs, and
\[
\pi=(\pi_j^{j+1}\colon(X_{j+1},f_{j+1})\to(X_j,f_j))_{j\ge1}
\]
is an inverse sequence of equivariant maps with MLC(1). Since $f$ is transitive, we have $X\in\mathcal{C}(f)$, so by Lemma 3.2, $X=[C_\ast]$ for some $C_\ast=(C_j)_{j\ge1}\in\mathcal{C}_\pi$. Then, any $x=(x_j)_{j\ge1}\in X$ satisfies $x_j\in C_j$ for all $j\ge1$, and $(X,f)$ is topologically conjugate to $\lim_{\pi'}(C_j,(f_j)|_{C_j})$, where
\[
\pi'=((\pi_j^{j+1})|_{C_{j+1}}\colon(C_{j+1},(f_{j+1})|_{C_{j+1}})\to(C_j,(f_j)|_{C_j}))_{j\ge1}.
\]
Note that $(C_j,(f_j)|_{C_j})$, $j\ge1$, are transitive SFTs, and by Lemma 3.4, $\pi'$ satisfies MLC(1); therefore, again without loss of generality, we may assume that $f_j$ is transitive for every $j\ge1$. Then, by Lemma 5.1, there is $D_\ast=(D_n)_{n\ge1}\in\mathcal{D}_\pi$ such that
\[
\tilde{\pi}=((\pi_j^{j+1})|_{D_{j+1}}\colon D_{j+1}\to D_j)_{j\ge1}
\]
satisfies MLC(1). Let $D=[D_\ast]\in\mathcal{D}(f)$. By Lemma 5.2, $D$ satisfies the conclusion of Lemma 5.3. Also, the conclusion of Lemma 5.4 is satisfied with $D$. Similarly as in the proof of Theorem 1.1 in \cite{Ka2}, we can show that
\[
D^n\cap{\rm DC1}_n^{\delta_n}(X,f)
\]
is a residual subset of $D^n$ for all $n\ge2$ for some $\delta_n>0$. Thus, by Lemma 5.5, we conclude that $D$ contains a dense Mycielski subset $S$, which is distributionally $n$-$\delta_n$-scrambled for all $n\ge2$, completing the proof. 
\end{proof}

\section{A remark on the chain components under shadowing}

Given any continuous map $f\colon X\to X$, $\mathcal{C}(f)$ can be seen as a quotient space of $CR(f)$ with respect to the closed $(f\times f)$-invariant equivalence relation $\leftrightarrow_f$ in $CR(f)^2$. Then, $\mathcal{C}(f)=CR(f)\slash\leftrightarrow_f$  is a compact metric space.

In the case of $\dim X=0$, if $f$ has the shadowing property, then by Lemma 1.2 and Lemma 2.4, $(X,f)$ is topologically conjugate to $\lim_\pi(X_n,f_n)$, where $(X_n,f_n)$, $n\ge1$, are SFTs, and
\[
\pi=(\pi_n^{n+1}\colon(X_{n+1},f_{n+1})\to(X_n,f_n))_{n\ge1}
\]
is an inverse sequence of equivariant maps with MLC(1). Without loss of generality, we consider
the case where $(X,f)=\lim_\pi(X_n,f_n)$. For any $C\in\mathcal{C}(f)$, by Lemma 3.2, we have $C=[C_\ast]$ for some $C_\ast=(C_n)_{n\ge1}\in\mathcal{C}_\pi$. As in the proof of Lemma 4.1, it can be shown that for each $N>0$, there is $C'_\ast=(C'_m)_{m\ge N}\in\prod_{m\ge N}\mathcal{C}(f_m)$ with the following properties:
\begin{itemize}
\item[(1)] $C'_N=C_N$,
\item[(2)] $\pi_m^{m+1}(C'_{m+1})\subset C'_m$ for every $m\ge N$,
\item[(3)] $\pi_m^{m+1}(C'_{m+1})=\pi_m^{m+2}(C'_{m+2})$ for all $m\ge N$.
\end{itemize}
Define $C''_\ast=(C''_n)_{n\ge1}\in\prod_{n\ge1}\mathcal{C}(f_n)$ by
\begin{equation*}
C''_n=
\begin{cases}
C_n&\text{if $1\le n<N$}\\
C'_n&\text{if $N\le n$}
\end{cases}
.
\end{equation*}
The properties (1) and (2) ensure $C''_\ast\in\mathcal{C}_\pi$. By Lemma 3.2, letting $C''=[C''_\ast]$, we obtain $C''\in\mathcal{C}(f)$, and by the property (3), similarly as in the proof of Lemma 4.1, it can be seen that $f|_{C''}\colon C''\to C''$ has the shadowing property. Note that for any neighborhood $U$ of $C$ in $CR(f)$, by the property (1) above, if $N$ is sufficiently large, then $C''\subset U$. Thus, letting
\[
\mathcal{C}_{sh}(f)=\{C\in\mathcal{C}(f)\colon\text{$f|_C$ has the shadowing property}\},
\]
we conclude that $\mathcal{C}_{sh}(f)$ is dense in $\mathcal{C}(f)$. In other words, we obtain

\begin{thm}
Let $f\colon X\to X$ be a continuous map with the shadowing property. If $\dim X=0$, then $\mathcal{C}(f)=\overline{\mathcal{C}_{sh}(f)}$.
\end{thm}

This theorem gives a positive answer to a question by Moothathu \cite{Moo} in the zero-dimensional case. Note that for any $C\in\mathcal{C}_{sh}(f)$, by the shadowing property of $f|_C$, we have $C=\overline{M(f|_C)}$, where $M(f|_C)$ denotes the set of minimal points for $f|_C$ (see \cite{Moo} for details).

As a complement to Theorem 6.1, we give an example of a continuous map $f\colon X\to X$ with the following properties:
\begin{itemize}
\item[(1)] $X$ is a Cantor space,
\item[(2)] $f$ has the shadowing property,
\item[(3)] $\mathcal{C}(f)$ is a Cantor space,
\item[(4)] $\mathcal{C}_{sh}(f)$ is a countable set and so is a meager subset of $\mathcal{C}(f)$.
\end{itemize}
 
\begin{ex}
\normalfont
For any closed interval $I=[a,b]$ and $c\in(0,1/2)$, let $\hat{I}=\{a,b\}$, $I_c^{(0)}=[a,a+c(b-a)]$, and $I_c^{(1)}=[b-c(b-a),b]$. Let $(c_j)_{j\ge1}$ be a sequence of positive numbers with $1/2>c_1>c_2>\cdots$. For any $s=(s_j)_{j\ge1}\in\{0,1\}^\mathbb{N}$, let
\[
i(s)=\bigcap_{j\ge0}I(s,j),
\]
where $I(s,j)$ is defined by $I(s,0)=[0,1]$, and $I(s,j+1)=I(s,j)_{c_{j+1}}^{(s_{j+1})}$ for every $j\ge0$. Let
\[
C=\{i(s)\colon s\in\{0,1\}^\mathbb{N}\}\subset[0,1]
\]
and note that $i\colon\{0,1\}^\mathbb{N}\to C$ is a homeomorphism, so $C$ is a Cantor space. For any $j\ge1$, let
\[
\hat{I}_j=\bigcup_{s\in\{0,1\}^\mathbb{N}}[I(s,j)]^{\hat{}}
\]
and note that $\hat{I}_1\subset\hat{I}_2\subset\cdots$. Also, let $A_1=\hat{I}_1$, $A_{j+1}=\hat{I}_{j+1}\setminus\hat{I}_j$, $j\ge1$, and
\[
A=\bigsqcup_{j\ge1}A_j\subset C.
\]

Let $\sigma\colon\{0,1\}^\mathbb{N}\to\{0,1\}^\mathbb{N}$ be the shift map. For each $k\ge1$, define
\[
\Sigma_k=\{x=(x_i)_{i\ge1}\in\{0,1\}^\mathbb{N}\colon\forall i\ge1,i+1\le\forall j\le i+k,x_i=1\Rightarrow x_j=0\},
\]
which is a mixing SFT, so $\sigma|_{\Sigma_k}\colon\Sigma_k\to\Sigma_k$ has the shadowing property. Note that $\Sigma_1\supset\Sigma_2\supset\cdots$ and consider
\[
\Sigma_\infty=\bigcap_{k\ge1}\Sigma_k=\{0^\infty,10^\infty\}\cup\{0^m10^\infty\colon m\ge1\}.
\]
Then, it is easily seen that $\sigma|_{\Sigma_\infty}\colon\Sigma_\infty\to\Sigma_\infty$ does not have the shadowing property. Let $\hat{\Sigma}_k=i(\Sigma_k)$, $k\ge1$, and $\hat{\Sigma}_\infty=i(\Sigma_\infty)$; here, $i\colon\{0,1\}^\mathbb{N}\to C$ is the homeomorphism defined above. Let
\[
X=[(C\setminus A)\times\hat{\Sigma}_\infty]\sqcup\bigsqcup_{k\ge1}[A_k\times\hat{\Sigma}_k],
\]
which is a perfect compact subset of $C\times C$ and so is a Cantor space.

Let $\hat{\sigma}=i\circ\sigma\circ i^{-1}\colon C\to C$ and
\[
f=(id_{C}\times\hat{\sigma})|_X\colon X\to X.
\]
To ensure the shadowing property of $f$, we define a sequence of positive numbers $(c_j)_{j\ge1}$ with $1/2>c_1>c_2>\cdots$ as follows. Fix a sequence of positive numbers $(\epsilon_k)_{k\ge1}$ with $\lim_{k\to\infty}\epsilon_k=0$. Denote by $\pi\colon X\to C$ the projection onto the first coordinate. Let
\[
B_k=\bigsqcup_{j=1}^k A_j,
\]
$k\ge1$. For each $k\ge1$, since $\pi^{-1}(B_k)$ is a finite disjoint union of SFTs,
\[
f|_{\pi^{-1}(B_k)}\colon\pi^{-1}(B_k)\to\pi^{-1}(B_k)
\]
has the shadowing property, implying the existence of $\delta'_k>0$ such that every $\delta'_k$-pseudo orbit $(x_i)_{i\ge0}$ of $f|_{\pi^{-1}(B_k)}$ is $\epsilon_k/2$-shadowed by some $x\in\pi^{-1}(B_k)$. Fix any $c_1\in(0,1/2)$ and assume that $c_k$, $k\ge1$, is given. For any $\delta_k\in(0,\epsilon_k/2)$, if $c_{k+1}\in(0,c_k)$ is small enough, then $X$ is contained in the $\delta_k$-neighborhood of $\pi^{-1}(B_k)$. Then, for every $\delta_k$-pseudo orbit $(y_i)_{i\ge0}$ of $f$, we have
\[
d(x_i,y_i)=\inf\{d(y_i,z_i)\colon z_i\in\pi^{-1}(B_k)\}\le\delta_k
\]
for all $i\ge0$ for some $x_i\in\pi^{-1}(B_k)$. Since
\[
d(f(x_i),x_{i+1})\le d(f(x_i),f(y_i))+d(f(y_i),y_{i+1})+d(y_{i+1},x_{i+1})
\]
for every $i\ge0$, if $\delta_k$ is small enough, then $(x_i)_{i\ge0}$ is a $\delta'_k$-pseudo orbit of $f|_{\pi^{-1}(B_k)}$, $\epsilon_k/2$-shadowed by some $x\in\pi^{-1}(B_k)$. This implies
\[
d(f^i(x),y_i)\le d(f^i(x),x_i)+d(x_i,y_i)\le\epsilon_k/2+\epsilon_k/2=\epsilon_k
\]
for all $i\ge0$, i.e., $(y_i)_{i\ge0}$ is $\epsilon_k$-shadowed by $x$. By defining $(c_j)_{j\ge1}$ in this way, we conclude that $f$ has the shadowing property.

Note that $X=CR(f)$ and
\[
\mathcal{C}_{sh}(f)=\{\pi^{-1}(u)\colon u\in A\},
\] 
which is a countable set. What is left is to show that $\mathcal{C}(f)$ is a Cantor space. Let $\pi_{\leftrightarrow_f}\colon X\to\mathcal{C}(f)$ be the quotient map. For any $x,y\in X$, we easily see that $x\leftrightarrow_f y$ iff $\pi(x)=\pi(y)$. This implies that there is a continuous map $h\colon\mathcal{C}(f)\to C$ with $\pi=h\circ\pi_{\leftrightarrow_f}$, which is bijective and so is a homeomorphism. Thus, $\mathcal{C}(f)$ is a Cantor space.
\end{ex}

\section*{Acknowledgements}

This work was supported by JSPS KAKENHI Grant Number JP20J01143.

\end{document}